\documentclass{amsart}

\usepackage{epsfig}
\usepackage{amsmath}
\usepackage{amssymb}
\usepackage{amsthm}
\usepackage{verbatim}
\usepackage{boxedminipage}
\usepackage[all, knot]{xy}
\xyoption{arc} 

\newtheorem{theorem}{Theorem}
\newtheorem{proposition}[theorem]{Proposition}
\newtheorem{corollary}[theorem]{Corollary}
\newtheorem{lemma}[theorem]{Lemma}

\newcommand{\G}{\mathbb{G}}
\newcommand{\A}{\mathbb{A}}
\newcommand{\C}{\mathbb{C}}
\newcommand{\Pj}{\mathbb{P}}
\newcommand{\Z}{\mathbb{Z}}

\newcommand{\Q}{\mathbb{Q}}
\newcommand{\Qbar}{\bar{\mathbb{Q}}}

\newcommand{\cL}{\mathcal{L}}
\newcommand{\cG}{\mathcal{G}}
\newcommand{\cW}{\mathcal{W}}
\newcommand{\cF}{\mathcal{F}}
\newcommand{\fF}{\mathfrak{F}}
\newcommand{\fq}{\mathfrak{q}}
\newcommand{\fw}{\mathfrak{w}}
\newcommand{\Gal}{\mathrm{Gal}\,}

\newcommand{\Spec}{\mathrm{Spec}\,}
\newcommand{\Prim}{\mathit{Prim}}

\newcommand{\Ind}{\mathrm{Ind}}
\newcommand{\ad}{\mathrm{ad}}

\newcommand{\notdiv}{\!\!\not|\,}
\newcommand{\WD}{\mathrm{WD}}
\newcommand{\Sp}{\mathrm{Sp}}
\newcommand{\Art}{\mathrm{Art}}
\newcommand{\GL}{\mathrm{GL}}
\newcommand{\SL}{\mathrm{SL}}
\newcommand{\PSL}{\mathrm{PSL}}
\newcommand{\PGL}{\mathrm{PGL}}
\newcommand{\ssm}{{^{\mathrm{ss}}}}
\newcommand{\into}{\hookrightarrow}
\newcommand{\Fbar}{\bar{F}}

\newcommand{\F}{\mathbb{F}}
\newcommand{\eps}{\epsilon}

\newcommand{\begstrat}{
\par\vspace{-0.3\baselineskip}\hfill\begin{boxedminipage}[t]{4.7in}
\hspace{-0.3in}\begin{minipage}{4.9in}
\begin{description}
}

\newcommand{\ra}{\rightarrow}

\title[Continuation for the zeta function of a Dwork hypersurface]{Meromorphic continuation for the zeta function of a Dwork hypersurface}
\author{Thomas Barnet-Lamb}
\email{tbl@math.harvard.edu} 
\address{Department of Mathematics\\Brandeis University\\415 South Street MS 050\\Waltham MA} 

\begin{document}
\thanks{The author was partially supported by NSF grant DMS-0600716 and by a Jean E.~de Valpine Fellowship.} 

\begin{abstract}
We consider the one-parameter family of hypersurfaces in $\Pj^5$ over $\Q$ with projective equation 
$$(X_1^{5}+X_2^{5}+X_3^{5}+X_4^{5}+X_5^{5}) = 5t X_1 X_2\dots X_5$$
(writing $t$ for the parameter), proving that the Galois representations attached
to their cohomologies are potentially automorphic, and hence that the zeta function of the family has meromorphic continuation throughout the complex plane.
\end{abstract}
\maketitle

\section{Introduction}
In \cite{hsbt}, Harris, Shepherd-Barron, and Taylor prove a potential modularity theorem, showing that certain Galois representations become automorphic after a sufficiently large totally real base change. In their argument, a key role is played by certain families of hypersurfaces, called Dwork families---in particular, by the part of the cohomology of the family which is invariant under a certain group action. (We will write $\fF$ for motive given by this part of the cohomology.) The importance of $\fF$ to their argument is reflected in the statement of the theorem they prove: in order to prove an $l$-adic Galois representation $r$ is potentially modular using their theorem, one requires, among other conditions, that  the restriction of the residual representation of $r$ to inertia at primes above $l$ be isomorphic to the restriction of the residual representation of some element of the family $\fF$.

They give two applications in their paper. On the one hand, through considerable ingenuity (and the fact that the Dwork family includes the Fermat hypersurface, whose cohomology restricted to inertia is easy to analyze) they are able to deduce that the odd symmetric powers of the cohomology of an elliptic curve over $\Q$ are modular, and (through further ingenuity) to deduce the Sato-Tate conjecture. On the other hand, the form of the condition on the inertial representation makes it very inviting to apply their modularity theorem to $\fF$ itself. It turns out to be fairly immediate that the other conditions of the potential modularity theorem are satisfied, and one can deduce the modularity of $\fF$ and hence the meromorphic continuation and functional equation of the zeta function of this part of the cohomology of the Dwork family.

A very natural question which then presents itself is the following: is it possible to gain enough control of the other parts of the cohomology of the Dwork family that one can prove meromorphic continuation for the \emph{whole} zeta function? In this paper, I answer this question in the affirmative for $N=5$, and also make some remarks on why further generalization of these methods is likely to be hard absent very significant advances in the technology of lifting theorems. In the analysis here, a key role is played by work of Katz, whose paper \cite{k} describes the relative cohomology sheaf of the Dwork family over the base, and its decomposition under the group action alluded to above, in very great detail.

{\bf Note added in proof:} Since this paper was written, the technology of potential automorphy has advanced somewhat, and some of the present argument will soon be able to be moderately simplified. In particular, the paper \cite{BLGGT} proves certain rather general potential automorphy theorems for regular, crystalline, self dual representations of the Galois group of totally real and CM fields. Using the results of this preprint, the argument for Proposition 5 of the present paper---which perhaps rather involved at present---could be replaced with an appeal to the theorems for compatible systems proved in the new manuscript. The remainder of this paper, in particular the analysis of the pieces of the cohomology in section 3, does not seem to be able to be simplified, even with these new results.

{\bf Acknowledgements:} I would like to thank my advisor, Richard Taylor, for suggesting this problem to me and for his help in pursuing it. I am grateful to Barry Mazur for helpful suggestions made concerning this material when he examined it as part of my PhD thesis. I also thank the anonymous referee for some further helpful suggestions.

\section{Dwork families discussed in detail}
In order to discuss these ideas more precisely, we must first describe Dwork families. (My notation for these families broadly follows Katz's paper \cite{k}, with $N$ in place of his $n$, \emph{except} that Katz works throughout with sheaves with coefficients $\Qbar_l$, whereas we will need the flexibility gained by working initially with $\Q_l$ coefficients, and extending to $\Qbar_l$ only as necessary to apply Katz's results. Our notation is not directly comparable with the notation of \cite{hsbt}.) Let $N$ be a positive integer. Fix a base ring $R_0=\Z[\frac{1}{N},\mu_N]$, where $\mu_N$ denotes the $N$th roots of unity. It is worth stressing this point: \emph{for the majority of this paper, we are working over $\Q(\mu_5)$ and all Galois representations are representations of subgroups of $G_{\Q(\mu_5)}$.} We will eventually return to working over $\Q$, but when we do so, this will be made explicit. We consider the scheme $Y$:
$$Y\subset \Pj^{N-1} \times \Pj^1$$
over $R_0$ defined by the equations
$$\mu(X_1^{N}+X_2^{N}+\dots+X_N^{N}) = N\lambda X_1 X_2\dots X_N$$
(using $(X_1:\dots:X_N)$ and $(\mu:\lambda)$ as coordinates on $\Pj^{N-1}$ and $\Pj^1$ respectively. We consider $Y$ as a family of schemes over $\Pj^1$ by projection to the second factor. We will label points on this $\Pj^1$ using the affine coordinate $t=\lambda/\mu$, and will write $Y_t$ for the fiber of $Y$ above $t$. From now on (apart from some remarks in the conclusion) we will be concerned exclusively with the case $N=5$.

The family $Y$ is smooth over the open set $U=\Spec R_0[\frac{1}{t^5-1}]$ away from the roots of unity. We are interested in the sheaf of relative cohomology of the family $Y$ above the set $U$. Let $l$ be a prime number which splits in $\Q(\mu_5)$.\footnote{I believe that this assumption could be dispensed with. However, we will only ever need the theory we are about to develop for one particular choice of $l$, and we will always be able to make this choice such that $l$ splits in $\Q(\mu_5)$. Therefore, I have chosen to make this assumption, since it simplifies the argument. In particular, it means that the sheaves $\Prim_l$ which we define later will have coefficient ring $\Q_l$, rather than an extension field.} Let $T_0=U[1/l]$, and form lisse sheaves
\begin{align}
\cF^i_l &:= R^i\pi_* \Q_l		\\
\cF^i[l] &:= R^i\pi_* \Z/l\Z
\end{align}
on $T_0$. (We remark that Katz's $\cF^i_l$ would be $\cF^i_l\otimes\Qbar_l$ in our notation, since he works with algebraically closed coefficients throughout.)
As a family of hypersurfaces, much of the cohomology of the $Y_\lambda$ is controlled by the hard Lefschetz theorem: for $i\neq N-2=3$, we have
$$
\cF^i_l = \begin{cases}
			0			&	\text{($i<0$)}							\\
			0			&	\text{($i>6$)}						\\
			0			&	\text{($i$ odd, $0\leq i\leq 6$)}			\\
			\Q_l(-j)		&	\text{($i=2j$ even, $0\leq i\leq 6$, $i\neq 3$)}
		\end{cases}
$$
The contribution to the zeta function from the characters $\Q_l(-j)$ is of course well understood. Thus in order to prove the functional equation for the zeta function of the whole variety, it suffices to control the zeta function of $\cF_l^3$. We will refer to the sheaf $\cF_l^3$ as $\Prim_l$ from now on. As discussed in the introduction, there is a natural group action on $\Prim_l$, allowing us to break down the cohomology into simpler pieces. Let us now introduce this group action. We will write $\Gamma$ for $(\mu_5)^5$, the $5$-fold product of the group of roots of unity, and $\Gamma_W$ for the subgroup of elements $(\zeta_1,\dots \zeta_5)$ with $\prod_{i=1}^5 \zeta_i=1$. $\Gamma_W$ acts on $Y$ with $(\zeta_1,\dots \zeta_5)$ acting via
$$((X_1:\dots:X_5),t)\mapsto ((\zeta_1 X_1:\dots:\zeta_5 X_5),t)$$
The image of $\mu_5$ embedded diagonally in $\Gamma$ lies in $\Gamma_W$ and acts trivially under this action. We will write $\Delta$ for this image. 

In \cite{hsbt}, the authors focus their attention on the invariants under this group action, a sheaf they refer to as $V$. They prove the following theorem:
\begin{theorem}[Theorem 4.4 of \cite{hsbt}] \label{hsbt-part-of-dwork-thm}
Suppose that $t\in \Q-\Z[1/5]$. Then the function $L(V_t,s)$ is defined and has meromorphic continuation to the whole complex plane, satisfying the functional equation
$$L(V,s) = \epsilon(V,s)L(V,4-s).$$
\end{theorem}

As I have said, our aim in this paper is to analyse the remaining parts of the cohomology and so establish the functional equation for the zeta function of the variety as a whole. As a first step to doing so, let us consider what other parts there actually are.

\section{The pieces of the cohomology}
The character group of $\Gamma$ is $(\Z/5\Z)^5$; that of $\Gamma_W$ is $(\Z/5\Z)^5 / \langle W\rangle$ where we write $W$ for the element $(1,1,\dots,1)$; and the character group of $\Gamma_W/\Delta$ is $(\Z/5\Z)_0^5 / \langle W\rangle$ where we write $(\Z/5\Z)_0^5$ for $\{(v_1,\dots,v_5)\in (\Z/5\Z)^5|\sum_i v_i = 0\}$. Thus the eigensheaves of $\Prim_l$ under the action described in the previous section are labeled by elements $v$ of $(\Z/5\Z)_0^5 / \langle W\rangle$: we may write such an element as $(v_1,\dots,v_5)$ mod $W$ with the $v_i$ elements of $\Z/5\Z$; it will often be convenient to abbreviate this to $[(v_1,\dots,v_5)]$. Note that our assumption that $\mu_5\in \Q_l$ is critical here in ensuring that the decomposition into eigensheaves is indeed defined with $\Q_l$ coefficients (and not with coefficient in some extension field). Note also that the labelling is not canonical, but depends on a choice of an identification of the copies of $\mu_5$ in $\Q_l$ and in $R_0$: equivalently, it depends on a choice of embedding $R_0\into\Q_l$. Having made such a choice, we shall write $\Prim_{l,[(v_1,\dots,v_5)]}$ for the piece of $\Prim_l$ where $\Gamma_W/\Delta$ acts via $[(v_1,\dots,v_5)]$. (Thus for instance $V_l=\Prim_{l,[(0,0,0,0,0)]}$.) Again, we remark that Katz's $\Prim_{l,v}$ would correspond to our $\Prim_{l,v}\otimes\Qbar_l$.

The obvious action of $S_5$ on $(\Z/5\Z)^5$ preserves $(\Z/5\Z)^5_0$ and $W$, and hence induces an action of $S_5$ on $(\Z/5\Z)_0^5 / \langle W\rangle$. Note that, if we permute the $(v_i)$ in this manner, the resulting sheaf $\Prim_{l,v}$ is isomorphic to the original  (the isomorphism being induced from the map on $Y$ which permutes the $X_i$ according to the same permutation). Thus to show that all the sheaves $\Prim_{l,v}$ are automorphic it will suffice to consider a set of $v$s which  represent all the orbits of $(\Z/5\Z)_0^5 / \langle W\rangle$ under $S_5$

\begin{proposition} \label{parts-prop}
The following list of $v$s represent all the orbits of $(\Z/5\Z)_0^5 / \langle W\rangle$ under $S_5$:
\begin{quotation}
[(0,1,2,3,4)], [(0,0,1,1,3)], [(0,0,1,2,2)], [(0,0,2,4,4)], [(0,0,3,3,4)], [(0,0,0,1,4)], [(0,0,0,2,3)], [(0,0,0,0,0)] 
\end{quotation}
Thus if $\Prim_{l,v}$ is automorphic for each of these $v$s, then it is for all $v$s.
\end{proposition}
\begin{proof}
We start with an arbitrary element $v$ of $(\Z/5\Z)_0^5 / \langle W\rangle$, and pick a representative $(v_1,\dots,v_n)\in (\Z/5\Z)_0^5$. By changing the representative of the congruence class mod $W$, we may ensure that in the list $(v_1,\dots,v_5)$, the 0 occurs at least as often as any other element of $\Z/5\Z$. Then, applying an appropriate permutation to $v$ (and hence to the $v_i$), we may ensure that the $v_i$ increase. (We order congruence classes mod 5 according to the order of their unique representatives in the range $0\dots 4$.)

Since 0 occurs at least as often as anything else, there must be at least one zero at the beginning of the list $(v_1,\dots,v_5)$. We split into several cases according to the number of zeroes there. It is trivial that if there is 1 zero then $v=[(0,1,2,3,4)]$, if there are $\geq4$ then $v=[(0,0,0,0,0)] $, and if there are 3 then the two remaining $v_i$ are 1 and 4 or 2 and 3.

If there are 2 zeroes, then we split into cases according to the value of $v_3$. If, for instance, $v_3=1$, then $v_4+v_5=4$, so $v_4,v_5$ must be $\{1,3\}$, $\{2,2\}$ or $\{0,4\}$, the last being impossible since the $v_i$ must increase. The other cases are similar.
\end{proof}

\begin{proposition} \label{littleprop} Assume we have chosen an arbitrary embedding of $R_0$ into $\Qbar_l$. For each $v$ in the following table, the dimension and Hodge-Tate numbers of $\Prim_{l,v}$ are as given:
\begin{center}
\begin{tabular}{|c|c|c|}
\hline
$v$		&$\dim \Prim_{l,v}$	& $HT(\Prim_{l,v})$	\\
\hline
$[(0,1,2,3,4)]$	&	0	&$\{\}$\\
$[(0,0,1,1,3)]$	&	2	&$\{1,2\}$\\
$[(0,0,1,2,2)]$	&	2	&$\{1,2\}$\\
$[(0,0,2,4,4)]$	&	2	&$\{1,2\}$\\
$[(0,0,3,3,4)]$	&	2	&$\{1,2\}$\\
$[(0,0,0,1,4)]$	&	2	&$\{1,2\}$\\
$[(0,0,0,2,3)]$	&	2	&$\{1,2\}$\\
\hline
\end{tabular}
\end{center}
(Thus, in particular, we note that $\Prim_{l,[(0,1,2,3,4)]}$ is zero-dimensional and that although the Hodge-Tate numbers depend in principle on the choice of embedding of  $R_0$ into $\Q_l$, in practice they are independent of this choice.)
\end{proposition}
\begin{proof}
Recall that at the beginning of this section we chose a particular embedding $R_0\ra\Q_l$ in order to label the pieces of the cohomology. (We remark that since $R_0=\Z[1/N,\mu_n]$ and we have a running assumption that $\mu_N\subset\Q_l$, this is the same thing as choosing an embedding $R_0\into\Qbar_l$.) Katz makes a corresponding choice in section 1 of \cite{k}, and the  Hodge-Tate numbers at this particular embedding (as well as the dimension, which does not depend on the choice of an embedding) may then be calculated by applying the procedure described in Lemma 3.1 of \cite{k}. (We will investigate what happens for Hodge-Tate numbers at the other embeddings later.) More precisely, Katz's procedure computes the Hodge-Tate numbers for \emph{his} sheaf $\Prim_{l,v}$, which is our $\Prim_{l,v}\otimes_{\Q_l}\Qbar_l$, but of course the Hodge-Tate numbers of $\Prim_{l,v}$ and $\Prim_{l,v}\otimes_{\Q_l}\Qbar_l$ are the same.

For instance, let us compute the dimension and Hodge-Tate numbers for $v=[(0,0,1,1,3)]$. We are asked to consider the coset of elements of $(\Z/5\Z)_0^5$ representing $v=[(0,0,1,1,3)]$, viz
the particular embedding $R_0\ra\Q_l$ which was chosen arbitrarily at the beginning of this section (or in section 1 of \cite{k}) and used to label the pieces of the cohomology
\begin{gather*}
\{(0,0,1,1,3), (1,1,2,2,4), (2,2,3,3,0), (3,3,4,4,1), (4,4,0,0,2)\}
\end{gather*}
The lemma then tells us that the dimension of $\Prim_l$ can be computed as the number of elements of this set which are \emph{totally nonzero}; that is, contain no 0s. There are two of these. Then, the Hodge-Tate numbers are computed by taking the degrees of the totally nonzero representatives above, where the \emph{degree} of an element $(v_1,\dots,v_n)\in (\Z/5\Z)_0^5$ is $\sum_i \tilde{v_i}$, and where (in turn) for each $i$, $\tilde{v_i}$ is the integer representing $v_i$ in the range 0 to $4$. Then the multi-set of these degrees is the multiset of Hodge-Tate numbers, with each element increased by 1. In our case, the HT numbers are therefore $\{\frac{1+1+2+2+4}{5}-1, \frac{3+3+4+4+1}{5}-1\}=\{1,2\}$.

For the other $v$s, the totally nonzero representatives are as shown in table \ref{nonzerorepstable},
\begin{table}
\begin{tabular}{|c|c|c|}
\hline
$v$		& totally nonzero representatives	\\
\hline
$[(0,0,1,2,2)]$		&\{(1,1,2,3,3), (2,2,3,4,4)\}\\
$[(0,0,2,4,4)]$		&\{(2,2,4,1,1), (4,4,1,3,3)\}\\
$[(0,0,3,3,4)]$		&\{(3,3,1,1,2), (4,4,2,2,3)\}\\
$[(0,0,0,1,4)]$		&\{(2,2,2,3,1), (3,3,3,4,2)\}\\
$[(0,0,0,2,3)]$		&\{(1,1,1,3,4), (4,4,4,1,2)\}\\
\hline
\end{tabular}
\caption{Totally nonzero representatives for certain $v$s. (See proof of Proposition \ref{littleprop}.)}
\label{nonzerorepstable}
\end{table}
and the result, for the Hodge-Tate numbers at our chosen embedding, follows.

Now, when we change our choice of embedding, the effect is to relabel the various pieces of the cohomology, by multiplying their labels $v$ by an element of $(\Z/5\Z)^\times$. It is easy to see, by inspection of table \ref{nonzerorepstable} above, that such relabeling sends an eigenspace to another eigenspace where the calculated Hodge-Tate numbers from the algorithm are the same. Whence we are done.
\end{proof}

\section{Controlling the $L$ functions}

We will now try to control the $L$ functions of the two-dimensional pieces we have singled out.
Before we go any further, we will need a little lemma
\begin{lemma} \label{lemma-introducing-d}
Let $v$ be taken from the table in the previous proposition. There is a constant $D$ such that if $M$ is an integer divisible only by primes $p>D$ and if $t\in U$ then the map
$$\pi_1(U,t) \ra SL(\Prim[M]_{v,t})$$
is surjective. (Here $SL(\Prim[M]_t)$ denotes the group of automorphisms of the 2 dimensional module $\Prim[M]_t$ with determinant 1.)
\end{lemma}
\begin{proof}
We first note that the monodromy of $\Prim_l\otimes\Qbar_l$ is Zariski dense in $\SL_2(\Prim_l)$, using Lemma 10.3 of \cite{k}, and remembering that $\Sp_2=\SL_2$. The same is then immediately seen to hold for $\Prim_l$. We can then deduce the result using Theorem 7.5 and Lemma 8.4 of \cite{mvw} or Theorem 5.1 of \cite{nori}.
\end{proof}

We now proceed to analyze the two dimensional pieces. We shall write $\Prim_{*,v}$ to mean the motive whose $l$ adic realizations are the $\Prim_{l,v}$s as $l$ varies. 
\begin{proposition} \label{two-dee} For each of the $v$ in the table above with $\Prim_{*,v}$ two dimensional, and for each $t \in \Q - \Z[1/10]$, we have that the function $L(\Prim_{*,v,t},s)$ is defined and has meromorphic continuation to the whole complex plane, satisfying the functional equation
$$L(\Prim_{*,v,t},s)=\epsilon(\Prim_{*,v,t},s)L(\Prim_{*,v^*,t},4-s)$$
where we write $v^*$ for $\{5-k|k\in v\}$.
\end{proposition}
Before we proceed to the proof, let us briefly remind ourselves of the significance of the words `is defined' in the statement of the theorem. The point is that, for each prime $p$, we wish to construct a local $L$ factor $L_p$, and we do so by looking at our motive's $l$-adic cohomology $\Prim_{l,v,t}$ for some $l\neq p$. Given an embedding $\Qbar\hookrightarrow\Qbar_l$, we can associate a Weil-Deligne representation 
$\WD(\Prim_{l,v,t}|_{\Gal(\Qbar_p/\Q_p)})^{\mathrm{F-ss}}$ 
to this $l$-adic cohomology at $p$, and to this, in turn, we can associate an $L$ factor. To get an unambiguous $L$ factor, we must insist that the Weil-Deligne representation (and hence the $L$ factor) do not depend on the choices we made: that is, the choice of $l$ and of an embedding $\Qbar\hookrightarrow\Qbar_l$. Thus the statement `$L(\Prim_{*,v,t},s)$ is defined' is saying that for every $p$, the local Weil-Deligne representation at $p$ constructed in this way is independent of these choices.

Our argument for the proposition draws heavily on \cite{hsbt}'s Theorem 3.3.

\begin{proof}[Proof of Proposition \ref{two-dee}] 
We first choose $q$ to be a rational prime dividing the denominator of $t$, so that $v_q(t) < 0$ and $q \notdiv 10$.

\emph{Step 1:} The goal of this step is to choose certain primes $l,l'$ which will be instrumental to the argument. 
In order to be in a position to do this we must first analyze the Zariski closure of the image of $\Gal(\Qbar/\Q(\mu_5))$ in the group $\GL(\Prim_{l,v,t})$ of automorphisms of the $\Q_l$ vector space $\Prim_{l,v,t}$. We will write $G_l$ for this image and $G_l^0$ for the connected component of the identity in it. 

By Lemma 10.1 of \cite{k}. the local monodromy of $\Prim_{l,v,t}\otimes\Qbar_l$ at $\infty$ is unipotent with a single Jordan block. (Condition 4 of the equivalent conditions there may be verified by direct inspection of each case in the table.) We immediately deduce the same for $\Prim_{l,v,t}$ itself. By the argument used to establish Lemma 1.15 of \cite{hsbt}, and recalling that $v_q(t) < 0$, we conclude that inertia at $q$ acts via a maximal unipotent. Thus $G_0^l$ contains such a maximal unipotent, and hence, by Proposition 3 of \cite{sch}, acts irreducibly. 

Moreover, the determinant map to $\G_m$ is dominating. To see this, note that Poincar\'e duality furnishes us with a perfect pairing between $\Prim_{l,v,t}$ and $\Prim_{l,v^*,t}$ towards $\Q(-3)$, and that $\Prim_{l,v^*,t}$ is the complex conjugate of $\Prim_{l,v,t}$. Thus we have that 
\begin{flalign} \label{eq-conj-self-dual}
\Prim_{l,v,t}^c = \Prim^\vee_{l,v,t} \epsilon_l^{-3}
\end{flalign}
which tells us, in turn that $(\det \Prim_{l,v,t})(\det \Prim_{l,v,t})^c = \epsilon_l^{-6}$, which would be impossible if the determinant character did not dominate $\G_m$.

Thus by Theorem 9.10 of \cite{k-gkm}, we may conclude that $G^0_l$ is $GL_2$. The main theorem of \cite{l} then 
tells us that the set of primes $l$ for which we fail to have
\begin{flalign} \label{eq-sandwich}
PSL_2(\Prim[l]_{v,t}) \subset \Gamma_l \subset PGL_2(\Prim[l]_{v,t}) 
\end{flalign}
has Dirichlet density 0. 
Next, observe that mimicing the argument for Proposition 3.4.2 of \cite{tbl-potmod}, we can construct a field $F^*(v)$ such that if a prime $l$ splits in $F^*(v)$, the natural polarization on $\Prim[l]_v$ coming from Poincare duality will have determinant a square. Now, since the set of primes for which we had equation (\ref{eq-sandwich}) had Dirichlet density 1, the set of primes for which we have equation (\ref{eq-sandwich}) and for which $l$ splits in all of $\Q(\mu_{10})$, $F^*(v)$ and $F^*(2,10)$ has positive density. (The field $F^*(2,10)$
is as defined in the statement of Proposition 3.4.2 of \cite{tbl-potmod}.)
We may therefore choose $l$ to be such a prime, and in addition insist that 
\begin{itemize}
\item $l> n, D, C(2,10)$ (the constant $C(2,10)$ was defined in \cite[Corollary 2.1.2]{tbl-potmod}); the constant $D$ is from Lemma \ref{lemma-introducing-d})
\item $v_{l}(t^5-1) = 0$
\end{itemize}
We choose $l'$ to be a distinct rational prime enjoying the same list of properties. Note that equation \ref{eq-sandwich} will ensure that the image of $\Gal(\Qbar/\Q(\zeta_l))$ in $\GL(\Prim[l]_t)$ is big via (say) Lemma 2.5.5 of \cite{cht}, and the simplicity of $\PSL_2(\F_l)$ will ensure that $\zeta_l\not\in \Qbar^{\ker \Prim[l]_t}$. 

\emph{Step 2.} Our next step in the proof is to establish that there exists a CM field $F_1/\Q(\mu_{10})$ and a $t'\in T(F_1)$ such that we have
\begin{align}
\Prim[l]_{v,t'} &\equiv \Prim[l]_{v,t} \label{property1}\\
\Prim[l']_{v,t'}|_{I_{F_{1,\fw}}} &\equiv \eps_l^{-1} \oplus \eps_l^{-2} \\
v_{\fq}(t') &<0 &&\fq|q	\\
v_{\fw}((t')^5-1) &=0 &&\fw|ll' \label{property4}
\end{align}
First, pick a point $t''\in \Q(\mu_{10})^+$ such that:
\begin{itemize}
\item If $\fw|ll'$ then $w((t'')^5 - 1)=0$
\item If $\fw|l'$ then $\Prim[l']_{t''}|_{I_\fw} \equiv \eps_l^{-1} \oplus \eps_l^{-2}$
\item $\Gal(\Qbar/\Q(\mu_{10})^+) \ra \SL(\Prim[l']_{t''})$ is surjective.
\end{itemize}
As in \cite{hsbt}, the existence of such a $t''$ relies on the form of Hilbert irreducibility with weak approximation; see \cite{e}. We may achieve the second condition by taking $t''$ to be $l'$-adically close to zero, since $\Prim[l']_{0}|_{I_w}$ is $\eps_l^{-1} \oplus \eps_l^{-2}$. (This last is because we know $\Prim[l']_{0}|_{I_w}$ to be a direct sum of characters, as in \cite{dmos}---and we know they are crystalline with Hodge-Tate numbers 1,2 from the table above.)

We introduce the character $\phi_l$ as is done in  \cite[\S3.2]{tbl-potmod}. (We will not follow \cite{tbl-potmod} in writing $`(\vec{h})$' for the twist by this character, since in fact $h(\sigma)=1$ for all $\sigma$ so one might think that `$(\vec{h})$' means `(1)', but this is \emph{not true}: $\phi^{-1}$ is not the cyclotomic character.) We will also consider the $l'$-adic version of $\phi$, as well as the mod $M:=ll'$ version; and we will abuse notation by writing $\phi$ for all of these.

Now, we follow the argument of the proof of Proposition 3.4.1 of \cite{tbl-potmod}, with the setup as follows: $\fq_j$ being the primes above $q$, $l_1=l,l_2=l'$, $\bar{\rho_1} = \Prim[l]_{v,t}\otimes\phi^{-1}$ and $\bar{\rho_2} = \Prim[l']_{v,t''}\otimes\phi^{-1}$, and $N=5$. In particular:
\begin{itemize}
\item We go through the first part of the argument in exactly the same way, introducing a mod
$M:=ll'$ character $\phi_l$ and a mod $M$ representation $\bar{\rho}_{\Z/M\Z}$. We see that
$\Prim[M]$ and $\bar{\rho}_{\Z/M\Z}$ become isomorphic once we disregard the Galois action
and keep only the modules with a pairing, using the assumption that $l$ splits in $F^*(v)$.
\item We next study the determinant $\det \Prim$. As in \cite{tbl-potmod}, from the fact that $\psi_1$ maps into the image of geometric monodromy, we can deduce that it is trivial, since we saw above that geometric monodromy was trivial. Thus we can deduce that $\det\Prim[l]_{s}\otimes\phi^{-1}$ is independent of $s$. The same holds for $l'$. (We also in fact have that, from an argument analogous to that establishing Lemma 3.2.1 of \cite{tbl-potmod}, that $(\det (\Prim[l]_s\otimes\phi^{-1}))=\epsilon_l^{-1}$.)
\item In particular, this tells us that $\det \bar{\rho_1}=\det(\Prim[l]_{v,t}\otimes\phi^{-1})$, and $\det \bar{\rho_2}=\det(\Prim[l']_{v,t''}\otimes\phi^{-1})$, and we can choose an isomorphism $\eta:\det(\Prim[M]\otimes\phi^{-1})\ra\det\bar{\rho_{\Z/M\Z}}$.
\item Then, in part (1) of the numbered list in the proof in \cite{tbl-potmod}, we have that the set of automorphisms which preserve this fixed isomorphism between determinants is $\SL(\Z/M\Z)$; we saw that the monodromy would be dense in this, since $l,l'>D$.
\item Finally, in part (3) of the numbered list, we see that the sets $\Omega_\fw$ are nonempty by observing that those above $l$ contain points above $t$ and those above $l'$ contain points above $t''$.
\end{itemize}
We then get that there is a CM field $F_1/\Q(\mu_{10})$ and a $t'\in T_{\cW}(F_1)$ satisfying the conditions \ref{property1}--\ref{property4} above. (For the second condition, that $\Prim[l']_{v,t'}|_{I_{F_\fw}} \equiv \eps_l^{-1} \oplus \eps_l^{-2}$, use the fact that $\Prim[l']_{v,t'}$ agrees with $\Prim[l']_{t''}$, which was chosen to have this property.)

\emph{Step 3.} Now, I claim that there exists a CM field $F/F_1/\Q(\mu_{10})$ such that $\Prim_{l',v,t'}\otimes\phi^{-1}|_{\Gal (\Qbar/F)}$ is automorphic, by appeal to Theorem 1.1.3 of \cite{tbl-potmod}, taking $\cL=\emptyset$, $N=10$, and $r=\Prim_{l,v,t'}\otimes\phi^{-1}$. 

Let us verify the conditions of this theorem in turn. We begin with the unnumbered conditions at the beginning
\begin{itemize}
\item $l$ splits in $\Q(\mu_{10})$; this is true by choice of $l$.
\item $q \notdiv 10$; this is true by choice of $q$. 
\end{itemize}
and then we address the numbered conditions
\begin{enumerate}
\item \emph{$r$ ramifies only at finitely many primes.} This is trivial, being true for all Galois representations which come from geometry.
\item $r^c \cong r^\vee\eps_l^{-1}$. For the same reason as eq \ref{eq-conj-self-dual} above, we have that.
$$\Prim_{l,v,t'}^c = \Prim^\vee_{l,v,t'} \epsilon_l^{-3}$$
whence we have what we want as $r=\Prim^\vee_{l,v,t'}\otimes\phi^{-1}$, and $\phi\phi^c=\epsilon_l^{-2}$.
\item \emph{The Bellaiche-Chenevier sign is +1.} This is because the Poincare duality pairing is symplectic and the multiplier of complex conjugation is odd.
\item \emph{$r$ is crystalline with the right Hodge-Tate numbers.} This follows immediately from the calculations of the previous proposition, once we note that the twist by $\phi$ changes the Hodge-Tate numbers by 1
\item \emph{$r$ is unramified at all the primes of $\cL$}. This is vacuous.
\item \emph{$r|_{\Gal(\Fbar_{v_q}/F_{v_q})}\ssm $ is unramified and  $r|_{\Gal(\Fbar_{v_q}/F_{v_q})} \ssm$ has Frobenius eigenvalues $1, (\#k(v_q)), \dots, (\#k(v_q))^{n-1}$}. By Lemma 10.1 of \cite{k}. the local monodromy of $\Prim_{l,v,t'}\otimes\Qbar_l$ at $\infty$ is unipotent with a single Jordan block. We immediately deduce the same for $\Prim_{l,v,t'}$. By the argument used to establish Lemma 1.15 of \cite{hsbt}, and recalling that $v_q(t') < 0$, we conclude that inertia at $q$ acts via a maximal unipotent and that the Frobenius eigenvalues are of the form required.
\item \emph{We have $\det r\equiv \eps_l^{-1}$} We saw above that $\det (\Prim[l]_s\otimes\phi^{-1})=\epsilon_l^{-1}$, as required.
\item \emph{Let $\bar{r}$ denote the semisimplification of the reduction of $r$, and $r'$ denote the extension of $r$ to a continuous homomorphism $\Gal(\Fbar/F^+)\ra\cG_n(\Qbar_l)$ as described in section 1 of \cite{cht}; then $\bar{r}'(\Gal(\Fbar/F(\zeta_l))$ is `big' in the sense of `big image'.} This is true by Lemma 2.5.5 of \cite{cht}, since we chose $t'$ such that $\Prim[l']_{t'}\equiv \Prim[l']_{t''}$, and we chose $t''$ such that $\Gal(\Qbar/\Q(\mu_{10})^+) \ra \GL(\Prim[l']_{t''})$ is surjective.
\item \emph{$\Fbar^{\ker\ad \bar{r}}$ does not contain $F(\zeta_l)$} This is true by the simplicity of $\PSL_2(\F_l)$ for $l>3$, again using the fact that $\Gal(\Qbar/\Q(\mu_{10})^+) \ra \GL(\Prim[l']_{t'})$ is surjective.
\item \emph{$r$ has the right restriction to inertia.} This was guaranteed by the choice of $t'$, once we note that the twist by $\phi$ changes restriction to inertia by $\eps_l$. 
\item \emph{We can choose a polarization with determinant a square.} This follows from the fact that $l'$
splits in $F^*(v)$.
\end{enumerate}
Having got that  $\Prim_{l',v,t'}\otimes\phi^{-1}|_{\Gal (\Qbar/F)}$ is also automorphic, we deduce $\Prim_{l,v,t'}\otimes\phi^{-1}|_{\Gal (\Qbar/F)}$ is automorphic since $Y_{t'}$ has good reduction at $l$, since we chose $t'$ such that $v_\fw((t')^5-1) = 0$ for $\fw$ over $l$. Whence also $\Prim_{l,v,t'}|_{\Gal (\Qbar/F)}$ is automorphic.

\emph{Step 4.} Next, I claim that $\Prim_{l,v,t}|_{\Gal(\Qbar/F)}$ is automorphic, by appeal to Theorem 4.3.4 of \cite{cht}. Let us verify the conditions of this theorem in turn:
\begin{enumerate}
\item $r^c \cong r^\vee\eps_l^{1-n}$. As for the corresponding condition of Theorem 1.1.3 of \cite{tbl-potmod}.
\item \emph{$r$ ramifies only at finitely many primes.} Again, this is trivial.
\item \emph{$r$ is crystalline.} As for the corresponding condition of Theorem 1.1.3 of \cite{tbl-potmod}.
\item \emph{Hodge-Tate numbers of $r$.} As for the corresponding condition of Theorem 1.1.3 of \cite{tbl-potmod}.
\item \emph{$r$ is discrete series somewhere.} We had above (in step 2) that inertia at $q$ acts via a maximal unipotent, which suffices.
\item \emph{$\Fbar^{\ker\ad \bar{r}}$ does not contain $F(\zeta_l)$.} True by the remarks immediately before step 2.
\item \emph{$\bar{r}'(\Gal(\Fbar/F(\zeta_l))$ is `big'.} True by the remarks immediately before step 2.
\item \emph{The residual representation is automorphic.} We have just verified that $\Prim[l]_{v,t'}$ is automorphic, and $\Prim[l]_{v,t'} \equiv \Prim[l]_{v,t}$.
\end{enumerate}

\emph{Step 5} We now use the following rather standard argument to deduce the functional equation of the $L$ function from the potential automorphy which we have just derived. As a virtual representation of $\Gal (F/\Q)$, we use Brauer's theorem to write
$$1=\sum_j a_j \Ind_{Gal(F/F_j)}^{\Gal{F/\Q}} \chi_j$$
where the $F_j$ are intermediate fields between $F$ and $\Q$ with $\Gal(F/F_j)$ soluble, the $a_j\in \Z$, and where for each $j$, $\chi_j: \Gal(F/F_j)\ra \C^\times$ is an isomorphism. By solvable base change, since $\Prim_{l,v,t}|_{\Gal(\Qbar/F)}$ is automorphic, so is $\Prim_{l,v,t}|_{\Gal(\Qbar/F_j)}$ for each $j$; that is, we can find a RAESDC representation $\pi_j$ of weight 0 and type $\{\Sp_n(1)\}_{\{v|q \}}$ such that for \emph{any} rational prime $l^*$ and isomorphism $\iota:\Qbar_l \overset\sim\to \C$ we have that 
$$r_{l^*,\iota}(\pi_j) \equiv \Prim_{l^*,v,t}|_{\Gal(\Qbar/F_j)}$$
Whence 
$$\Prim_{l^*,v,t} = \sum_j a_j \Ind_{Gal(F/F_j)}^{\Gal{F/\Q}} r_{l^*,\iota}(\pi_j \otimes (\chi_j\circ\Art_{F_j}))$$
We deduce, using Theorem 3.2 and Lemma 1.3(2) of \cite{ty}, that the $L$ function of $\Prim_{*,v,t}$ is defined and that 
$$L(\Prim_{*,v,t})=\prod_j L(\pi_j \otimes (\chi_j\circ\Art_{F_j}),s)^{a_j}$$
which gives the result, since each of the multiplicands on the right hand side obey the expected functional equation, whence the left hand side does too.
\end{proof}

We can now put together what we know so far and control the overall $L$ function of the $\Prim_l$.

\begin{corollary} The $L$ function of $\Prim_l$ has meromorphic continuation to the whole complex plane, for $t \in \Q - \Z[1/10]$.
\end{corollary}
\begin{proof}
Proposition \ref{parts-prop} gives us a list of pieces whose $L$ functions we must control. We may control $\Prim_{l,[(0,0,0,0,0)]}$ by Theorem \ref{hsbt-part-of-dwork-thm} and the rest by Proposition \ref{two-dee}.
\end{proof}

As we have set things up, the sheaf $\Prim_l$ has base defined over $\Q(\mu_5)$; but it could also have been defined over $\Q$ (unlike the various pieces $\Prim_{l,v}$, most of which are not defined over $\Q$---we have that $\Gal(\Q(\mu_5)/\Q)$ intermixes the various pieces). From now on, we will consider $\Prim_l$ to have been defined over $\Q$ and recapitulating the last part of the proof of Proposition \ref{two-dee} gives us:

\begin{theorem} The $L$ function of $\Prim_{*,t}$ (now considered to be defined over $\Q$) has meromorphic continuation to the whole complex plane, for $t \in \Q - \Z[1/10]$.
\end{theorem}
\begin{proof}
By steps 1--4 of the proof of Proposition \ref{two-dee}, there are fields $F^{(v)}$ such that $\Prim_{l,v,t}|_{\Gal(\Qbar/F^{(v)})}$ is automorphic for each $v$ in the table given in Proposition \ref{littleprop}; by the proof of Theorem \ref{hsbt-part-of-dwork-thm} given in \cite{hsbt} the same is true for $v=(0,0,0,0,0)$, and by Proposition \ref{parts-prop}, the symmetry of the situation allows us to deduce this for all other $v$. We can modify the proofs of these theorems to ensure that a single field extension $F$ makes all of these representations automorphic simultaneously. (For instance, the proof of Theorem 3.1 of \cite{hsbt} can handle multiple representations simultaneously; there are no essential difficulties other than those of bookkeeping.). Then the whole sheaf $\Prim_{l,t}$ becomes automorphic when restricted to $G_F$.

We can then use the argument of step 5 of the proof of Proposition \ref{two-dee} with $\Prim_{*,t}$ taking the place of $\Prim_{*,v,t}$ to deduce the expected functional equation for $L(\Prim_{*,t},s)$ and thus meromorphic continuation.
\end{proof}

\begin{corollary} The zeta function of $Y_t$, for $t \in \Q - \Z[1/10]$, has meromorphic continuation.
\end{corollary}
\begin{proof}
By the remarks preceding Theorem \ref{hsbt-part-of-dwork-thm}, the remaining parts of the cohomology are well understood using the hard Lefschetz theorem. 
\end{proof}

\section{Concluding remarks}
We have seen that the zeta function of the hypersurface with projective equation
$$(X_1^{5}+X_2^{5}+X_3^{5}+X_4^{5}+X_5^{5}) = 5t X_1 X_2\dots X_5$$
has a meromorphic continuation and satisfies the expected functional equation. It is perhaps natural to wonder whether the techniques used might generalise to more general hypersurfaces of a similar type. For instance, the paper \cite{hsbt} shows that the $\Gamma_W/\Delta$ invariants in the cohomology of the variety:
\begin{equation} \label{more-general-eq}
(X_1^{N}+X_2^{N}+X_3^{N}+X_4^{N}+\dots+X_N^{N}) = Nt X_1 X_2\dots X_N
\end{equation}
will be automorphic for all odd $N$, so we might wonder whether the result of this paper can be generalised to other $N$s. The paper \cite{k} works in an even more general context, replacing the monomial $X_1 X_2\dots X_N$ on the RHS of the defining equation with an arbitrary monomial of the required degree, so one might also ask if there are any cases of that form to which we might try to generalise the result of this paper. I feel that a few remarks on these cases may be useful to the reader.

\subsection{Smaller $N$s in equation \ref{more-general-eq}} It is worth beginning by noting that the cases $N=1,2$ are trivial, and the case $N=3$ is also uninteresting since then equation \ref{more-general-eq} describes a  family of elliptic curves, and the zeta function is already understood. Thus the only interesting case with smaller $N$ is $N=4$. 

At first sight, it might seem difficult to analyse this case using the methods of this paper, since the result \cite{hsbt} of Harris, Shepherd-Barron and Taylor which gives the automorphicity of $\Prim_{l,[(0,\dots,0)]}$ requires $N$ to be odd. But the paper \cite{tbl-odd} generalises their methods to cover odd-dimensional cases, and it is then possible to extend the methods of this paper to cover that case, too. In particular, an analysis like that in section 3 of the present paper will reveal that all the pieces of the cohomology apart from $\Prim_{l,[(0,0,0,0)]}$ are one- or zero-dimensional, and so trivially automorphic.

I have chosen not to give this argument in full detail, since a very beautiful geometric argument of Elkies and Sch\"utt \cite{es} tells us that, for the $N=4$ case, each Dwork hypersurface is isogenous to the Kummer surface of a product $E_1\times E_2$, where $E_1$ and $E_2$ are elliptic curves defined over a quadratic extension of $\Q$, conjugate to each other over $\Q$ and related by a 2-isogeny.

This allows one to quite directly see the automorphicity required in this case, and it seemed that little would be served by giving the full details of the argument above.

\subsection{Larger $N$s in equation \ref{more-general-eq}} If we try to extend the methods of this paper to larger values of $N$, we face the following problem

\begin{proposition} \begin{enumerate} \item Let $N\geq 8$ be an integer. Then the Hodge-Tate numbers of 
$\Prim_{l,[(4,N-2,N-2,0,\dots,0)]}$ include 2 with multiplicity at least 2.

\item Let $N=6$. Then the Hodge-Tate numbers of $\Prim_{l,[(0,0,0,2,2,2)]}$ include 3 with multiplicity at least 2.

\end{enumerate}
\end{proposition}
\begin{proof}
Again, we use Lemma 3.1 of \cite{k}. For point 1, the totally nonzero representatives include both 
$(5,N-1,N-1,1,\dots,1)$ and $(7,1,1,3,\dots,3)$, and 
$$\frac{5+(N-1)+(N-1)+1+\dots+1}{N} - 1 =\frac{7+1+1+3+\dots+3}{N} - 1 = 2$$
so 2 occurs as a Hodge-Tate weight with multiplicity at least 2. Part 2 is similar; the totally nonzero representatives include both $(2,2,2,4,4,4)$ and $(5,5,5,1,1,1)$, and 
$$\frac{2+2+2+4+4+4}{6} - 1 =\frac{5+5+5+1+1+1}{6} - 1 = 3$$
and the result follows. \end{proof}

Thus in all cases with $N$ even (recall that we need $N$ even for \cite{hsbt} to apply\footnote{It is worth remarking that even if this were not an obstacle, the $N=7$ case also has a piece of the cohomology with a repeated Hodge-Tate number.}) and $N\geq6$, at least one of the pieces of the cohomology of eq \ref{more-general-eq} will have a repeated Hodge-Tate number. At present, apart from some work in the case of two-dimensional Galois representations, there are no modularity lifting theorems for representations with repeated Hodge-Tate numbers, and hence (since such theorems are a key ingredient in proving the potential modularity theorems such as \cite{t-potmod} and \cite{hsbt} on which this paper relies) it seems unlikely that the approach of this paper can be extended to cover such cases.

(One might briefly wonder whether some larger algebra of correspondances could be used to cut the cohomology into smaller pieces, small enough that they no longer have repeated Hodge-Tate weights; but this is impossible, since the results of Katz on the monodromy of the cohomology tell us that all the pieces in the decomposition of the cohomology into eigenspaces for $\Gamma_W/\Delta$ cannot be broken up further, as the monodromy acts transitively on each piece.)

\subsection{Other RHS monomials in equation \ref{more-general-eq}} Katz studies the more general equation:
\begin{equation} \label{even-more-general-eq}
(X_1^{N}+X_2^{N}+X_3^{N}+\dots+X_N^{N}) = N\lambda \prod_{i} X_i^{w_i}
\end{equation}
where $W=(w_1,\dots,w_N)$ is a sequence of non-negative integers summing to $N$. It is natural to ask whether the methods of this paper can be extended to any varieties of this form, beyond the cases already considered. Unfortunately, the answer is no. 

Let us imagine how an analysis based on the same techniques as those used above would go. As before, the main challenge would be to analyze the middle dimensional cohomology, since the rest is determined by hard Lefschetz. We can define $\Prim^{N-2}_l$, as in \cite{k}, to correspond to the part of the middle-dimensional cohomology not coming from Lefschetz. Following the method above, our next step is to decompose this cohomology into eigensheaves. 

The natural group acting on eq \ref{even-more-general-eq} is easily seen to be 
$$\left\{(\zeta_1,\dots,\zeta_N) \in (\mu_N)^N \Big| \prod \zeta_i^{w_i} = 1\right\} / \Delta$$
where $\Delta$, as before, is $\mu_N$ embedded diagonally. This has character group $(\Z/N\Z)_0/\langle W \rangle$, where we abuse notation by considering $W$ as a class in $(\Z/N\Z)_0$. We will write an element of $(\Z/N\Z)_0/\langle W \rangle$ as either $v$ mod $W$ or simply $[v]$ and define $\Prim_{l,[v]}^{N-2}$ in a similar manner to before. 

Suppose now that we have fixed some $W$. The main challenge in applying the methods to this paper to show that the zeta function of the family (\ref{even-more-general-eq}) is meromorphic will be showing that $\Prim_{l,[v]}^{N-2}$ is automorphic for each $v$. Since this will rely, in the final analysis, on the application of a lifting theorem, we will certainly require that $\Prim_{l,[v]}^{N-2}$ has distinct Hodge numbers for all $v$. This, unfortunately, will never happen except in the cases already considered.

\begin{proposition} Suppose $N\geq 3$ and $W\neq(1,1,\dots,1)$. Then there exists some $v$ such that the Hodge-Tate numbers of $\Prim_{l,[v]}^{N-2}$ are not all distinct. 
\end{proposition}
\begin{proof}
Since $W\neq(1,1,\dots,1)$, some $w_i$ (wlog $w_1$) is 0. For $i>1$, let us write $h_i$ for $\mathrm{hcf}(w_i,N)$. For each $i>1$ where $h_i=1$, set $v_i=0$; then, for $k \in \Z$
$$v_i + k w_i  \equiv 0\quad\text{(mod $N$)}\qquad\text{if and only if}\qquad k\equiv0\quad\text{(mod $N$)}$$
For $i>1$ where $h_i>1$, we can choose $v_i=1$; then $v_i+kw_i$ will always be $\equiv 1$ mod $h_i$, and we will \emph{never} have $v_i + k w_i  \equiv 0$ mod $N$. Finally, $v_1$ is fixed by the condition that the $v_i$ sum to 0. (Note that $\tilde{v}_1>1$, since if not $h_i=1$ for all $i>1$, hence $v_i>1$ for all such $i$, which is impossible.)

The elements of $(\Z/N\Z)_0$ representing $[v]$ are $v+kW$ for $k\in\Z/N\Z$; it is immediate that the element $v+kW$ is totally nonzero for all $k\in\Z/N\Z$ except $k=0$, so the totally nonzero representatives are $\{v,v+W,\dots,v+(N-1)W\}$, and the multiset of Hodge-Tate numbers is the multiset
$$\{\deg(v)-1,\deg(v+W)-1,\deg(v+2W)-1,\dots,\deg(v+(N-1)W-1\}$$

If we suppose for contradiction that these numbers are distinct, we have that the N-1 integers $\deg(v+W),\deg(v+2W),\dots,\deg(v+(N-1)W)$ are distinct. We now note that the degree of an element of $(\Z/N\Z)_0$ is trivially $\leq N-1$, and that writing $u$ for some $v+kW$
\begin{align*}
N \deg u&\geq\tilde{u}_1 + (N-1)	\quad\text{(since $u$ totally nonzero)}\\
&=\tilde{v}_1 + (N-1) > N	\quad\text{(since $w_1=0$)}
\end{align*}
whence $\deg (v+kW)>1$; so we have $N-1$ distinct integers $\deg (v+kW)$ with $1<\deg (v+kW)\leq N-1$, a contradiction.
\end{proof}

\end{document}